\documentclass[envcountsect]{svjour3}
\usepackage{amsmath,amsxtra, amssymb, latexsym, amscd,cite,comment}
\usepackage{graphicx,color}
\usepackage[active]{srcltx}
\usepackage[utf8]{inputenc}
\usepackage[mathscr]{euscript}
\usepackage{mathrsfs}
\usepackage[english]{babel}
\usepackage{pgf,tikz}
\usetikzlibrary{arrows}
\usepackage{latexsym,amsfonts,verbatim,cite,marvosym,enumerate}
\usepackage{hyperref}
\smartqed

\begin{document}
	
	
	\title{Optimality conditions for approximate Pareto solutions of a nonsmooth  vector optimization problem with an infinite number of constraints}
	
	\titlerunning{Optimality conditions for approximate Pareto solutions of problem \eqref{semi-problem}}
	\author{Ta Quang Son$^1$ \and Nguyen Van Tuyen$^{2}$ \and  Ching-Feng Wen$^3$
	}
	\authorrunning{T.Q. Son, N.V. Tuyen, and C.-F. Wen}
	
	\institute{\Letter\ \ Nguyen Van Tuyen
			 \\
		\text{\ \ \ \ }	{tuyensp2@yahoo.com; nguyenvantuyen83@hpu2.edu.vn}
	\\
	\\
	\text{\ \ \ \ } Ta Quang Son
			\\
		\text{\ \ \ \ } {taquangson@sgu.edu.vn}
		\\
		\\
		\text{\ \ \ \ } Ching-Feng Wen
		\\
		\text{\ \ \ \ } {cfwen@kmu.edu.tw}
		\\
		\\
		\at$^1$\text{\ \ \ \ }Faculty of Mathematics and Applications, Saigon University, Hochiminh City, Vietnam 
		\and 
		\at$^2$\text{\ \ \ \ }Department of Mathematics, Hanoi Pedagogical University 2, Xuan Hoa, Phuc Yen, Vinh Phuc, Vietnam 
		\and
		\at$^3$\text{\ \ \ \ }Center for Fundamental Science; and Research Center for Nonlinear Analysis and Optimization, Kaohsiung Medical University; Department of Medical Research, Kaohsiung Medical University Hospital, Kaohsiung, 80708, Taiwan  		
	}

\date{Received: date / Accepted: date}
\maketitle
\begin{abstract}  In this paper, we present some new necessary and sufficient optimality conditions in terms of the Clarke subdifferentials for approximate Pareto solutions of a nonsmooth vector optimization problem which has an infinite number of constraints. As a consequence, we obtain optimality conditions for  the particular cases of cone-constrained convex vector optimization problems and semidefinite vector optimization problems. Examples are  given to illustrate the obtained results. 
\end{abstract}
\section{Introduction} \label{introduction}
 This paper mainly deals with constrained vector optimization problems formulated as follows:
\begin{equation}\label{semi-problem}\tag{VP}
\begin{split}
&\text{\rm Min}_{\,\mathbb{R}^m_+}\,f(x):=(f_1(x), \ldots, f_m(x))
\\
&\text{s.t.}\ \ x\in C:=\{x\in \Omega \;|\; g_t(x)\leqq 0,\ \ t\in T\},  
\end{split} 
\end{equation}
where $f_i$, $i\in I:=\{1, \ldots, m\}$, and  $g_t$, $t\in T$, are locally Lipschitz functions from a Banach space $X$ to $\mathbb{R}$, $\Omega $ is a  nonempty and closed subset of $X$, and  $T$  is an arbitrary (possibly infinite) index set. Optimization problems of this type relate to
{\em semi-infinite vector optimization problems}, provided that the space $X$ is finite-dimensional, and to {\em infinite vector optimization problems} if $X$ is infinite-dimensional; see \cite{Bonnans,Gorberna}. The modeling of problems as \eqref{semi-problem} naturally arises in a wide range of applications in various fields of mathematics, economics and engineering; we refer the readers to the books \cite{Gorberna,Reemtsen} and to the papers \cite{Canovas,Caristi,Chuong09,Chuong092,Chuong10,chuong14,Dinh07,Dinh09,Hettich,Kanzi10} with the references therein. 

Our main concern is to study optimality conditions for approximate Pareto solutions of problem \eqref{semi-problem}. It should be noted here that the study of approximate solutions is very significant because, from the computational point of view, numerical algorithms usually generate only approximate solutions if we stop them after a finite number of steps. Furthermore, the solution set may be empty in the general noncompact case (see \cite{Huy17,Jahn,Kim18,KMPT,Luc,Tuyen16}) whereas approximate solutions exist under very weak assumptions (see Propositions \ref{Theorem33} and \ref{Theorem34} in Section \ref{Section2} below).   

In the literature, there are many publications devoted to optimality conditions for approximate solutions  of  semi-infinite/infinite scalar optimization problems; see, for example,  \cite{Dinh07,Kanzi10,Kanzi14,L114,Long18,Loridan1982,Mishra12,Son09}. However, in contrast to the scalar case, there are a few works dealing optimality conditions for approximate Pareto solutions of semi-infinite/infinite vector optimization problems; see \cite{Lee12,Son18,Shitkovskaya18}. In \cite{Lee12,Shitkovskaya18}, the authors  obtained necessary and sufficient optimality conditions  for approximate Pareto solutions of a convex semi-infinite/infinite vector optimization problem under various kind of Farkas--Minkowski constraint qualifications. By using the Chankong--Haimes scalarization method,  Kim and Son \cite{Son18} established some necessary  optimality conditions for approximate quasi Pareto solutions of  a locally Lipschitz semi-infinite vector optimization problem. 

In this paper, we present some necessary conditions of Karush--Kuhn--Tucker  type for approximate (quasi) Pareto solutions of  problem \eqref{semi-problem}  under a Slater-type constraint qualification hypothesis. Sufficient optimality conditions for approximate (quasi) Pareto solutions of problem \eqref{semi-problem} are also provided by means of introducing the concepts of (strictly) generalized convex functions, defined in terms of the Clarke subdifferential of locally Lipschitz functions. The obtained results improve the corresponding ones in \cite{Lee12,Son18,Shitkovskaya18}. As an application, we establish optimality conditions for cone-constrained convex vector optimization problems and semidefinite vector optimization problems.  In addition, some examples are also given for  illustrating the obtained results. 

In Section \ref{Section2}, we  recall some basic  definitions and preliminaries from the theory of vector optimization and variational analysis. Section \ref{section3}  presents the main results. The application of the obtained results in Section \ref{section3} to cone-constrained convex vector optimization problems and semidefinite vector optimization problems is addressed in Section \ref{Applications}.  
\section{Preliminaries} \label{Section2}

\subsection{Approximate Pareto solutions} 
Let $\mathbb{R}^m_+ := \{y := (y_1, \ldots, y_m)\,|\, \, y_i\geqq 0,\,\, i\in I\}$ be the nonnegative orthant in $\mathbb{R}^m$. For $a, b\in\mathbb{R}^m$, by $a\leqq b$, we mean $b-a\in \mathbb{R}^m_+$; by $a\leq b$, we mean $b-a\in \mathbb{R}^m_+\setminus\{0\}$; and by $a<b$, we mean $b-a\in \text{int}\,\mathbb{R}^m_+$. 
  
\begin{definition}[{see \cite{Loridan1984}}]{\rm   
   Let $\xi \in \mathbb{R}^m_+$ and $\bar x\in C$. We say that
\begin{enumerate}[\rm(i)]
\item $\bar{x}$ is a {\em weakly Pareto solution} (resp.,  a {\em Pareto solution}) of  \eqref{semi-problem} if there is no $x\in C$ such that 
$$f(x)<  f(\bar{x})\ \ (\text{resp.}, \,\, f(x)\leq  f(\bar{x})).$$

\item  $\bar x$ is a {\em $\xi$-weakly Pareto solution} (resp., a {\em $\xi$-Pareto solution}) of  \eqref{semi-problem} if there is no $x\in C$ such that 
$$f(x) + \xi  <  f(\bar x) \ \ (\text{resp.}, \,\,  f(x) + \xi  \leq  f(\bar x)).$$

\item  $\bar x$ is a {\em $\xi$-quasi-weakly Pareto solution} (resp., a {\em $\xi$-quasi Pareto solution}) of  \eqref{semi-problem} if there is no $x\in C$  such that 
$$f(x) + \|x - \bar x\| \xi  <  f(\bar x)\ \ (\text{resp.}, \,\, f(x) + \|x-\bar x\|\xi\leq f(\bar x)).$$
\end{enumerate}	
 
}\end{definition}

\begin{remark}{\rm   
 If $\xi=0$, then the concepts of a $\xi$-Pareto solution and a $\xi$-quasi-Pareto solution (resp., a $\xi$-weakly Pareto solution and a $\xi$-quasi-weakly Pareto solution)  coincide with the concept of  a Pareto solution (resp., a  weakly Pareto solution). Hence when dealing with approximate Pareto solutions we only consider the case that $\xi\in \mathbb{R}^m_+\setminus\{0\}$.
}\end{remark}

\begin{definition}{\rm 
Let $A$ be a subset in $\mathbb{R}^m$ and ${\bar{y}} \in \mathbb{R}^m$. The set $A \cap ({\bar{y}} - \mathbb{R}^m_+)$ is called a {\em section} of $A$ at ${\bar{y}}$ and denoted by $[A]_{\bar{y}}.$ The section $[A]_{\bar{y}}$ is said to be {\em bounded} if, and only if, there is $a\in\mathbb{R}^m$ such that
$$[A]_{\bar{y}} \subset a + \mathbb{R}^m_+.$$ 
}\end{definition}

\begin{remark}\label{remark-2}{\rm Let $\bar y$ be an arbitrary element in $f(C)$. It is easily seen that every $\xi$-Pareto solution (resp., $\xi$-quasi Pareto solution) of \eqref{semi-problem} on $C\cap  f^{-1}\left([f(C)]_{\bar y}\right)$ is also a  $\xi$-Pareto one (resp., $\xi$-quasi Pareto one) of \eqref{semi-problem} on $C$.     
}
\end{remark}
The following results  give some sufficient conditions for the  existence of approximate Pareto solutions of problem \eqref{semi-problem}.

\begin{proposition}[Existence of approximate Pareto solutions] \label{Theorem33} 
Assume that $f(C)$ has a nonempty bounded section. Then, for each $\xi \in \mathbb{R}^m_+\setminus\{0\},$ the problem~\eqref{semi-problem} admits at least one $\xi$-Pareto solution.
\end{proposition}
\begin{proof} The assertion follows directly   from Remark \ref{remark-2} and \cite[Lemma 3.1]{ha06}, so is omitted. $\hfill\Box$
\end{proof}
\begin{proposition}[Existence of approximate quasi Pareto solutions]  \label{Theorem34} 
If $f(C)$ has a nonempty bounded section, then, for every $\xi \in {\rm int}\, \mathbb{R}^m_+,$   problem \eqref{semi-problem} admits at least one $\xi$-quasi Pareto solution.
\end{proposition}
\begin{proof} 
Let $x^0 \in C$ be such that the section of $f(C)$ at $f(x^0)$ is bounded. By Proposition~\ref{Theorem33}, there exists $\bar x \in C$ such that $f(\bar x)\leqq  f(x^0)$ and 
\begin{eqnarray*} 
f(C)\ \cap\ [f(\bar x) - \xi - \mathbb{R}^m_+\setminus\{0\}] &=& \emptyset.
\end{eqnarray*}
Consequently, 
$$f(C)\cap [f(\bar x) - \xi - \mathrm{int}\, \mathbb{R}^m_+]\ =\  \emptyset.$$
By the continuity of $f$ and the closedness of $C,$  for each $x \in C,$ the set
$$\{u \in C \; | \; f(u)+\|u-x\|\xi \ \leqq  \ f(x)\}$$
is closed. By \cite[Theorem 3.1]{araya}, there exists $x^*\in C$ such that $f(x^*)<f(\bar x)$ and 
\begin{eqnarray*}
f(x) + \|x-x^*\|\xi- f(x^*) & \notin& -\mathbb{R}^m_+ \quad \textrm{for all} \quad x \in C, x\neq x^*.
\end{eqnarray*}
Consequently, 
\begin{eqnarray*}
f(x) + \|x-x^*\|\xi & \nleq & f(x^*) \quad \textrm{ for all } \quad x \in C, x \ne x^*.
\end{eqnarray*}
Thus,  $x^*$ is a $\xi$-quasi Pareto solution of \eqref{semi-problem}. The proof is complete.  $\hfill\qed$
\end{proof}
\subsection{Normals and subdifferentials}  For the Banach space $X$, the bracket $\langle \cdot\,, \cdot\rangle$ stands for the canonical pairing between space $X$ and its dual $X^*$.  The  closed unit ball of $X$ is denoted by $B_X.$ The closed ball with center $x$ and radius $\delta$ is denoted by $B(x, \delta)$. Let $A$ be a nonempty subset of $X$. The topological interior, the topological closure, and the convex hull  of A are denoted, respectively, by $\mathrm{int}\, A$, $\mathrm{cl}\, A$, and $\mathrm{cone}\, A$. The symbol $A^\circ$ stands for the polar cone of a given set $A\subset X$, i.e.,
\begin{equation*}
A^\circ=\{x^*\in X^*\;|\; \langle x^*, x\rangle\leqq 0, \ \ \forall x\in A\}.
\end{equation*}
 
Let $\varphi \colon X \to  \mathbb{R}$ be a locally Lipschitz function. The Clarke generalized directional derivative  of $\varphi$ at $\bar x\in X$ in the direction $d\in X$, denoted by $\varphi^\circ(\bar x; d)$, is defined by
\begin{equation*}
\varphi^\circ(\bar x; d):=\mathop{\limsup\limits_{x\to\bar x}}
	\limits_{t\downarrow 0}\frac{\varphi(x+td)-\varphi(x)}{t}.
\end{equation*}
The Clarke subdifferential of $\varphi$ at $\bar x$ is defined by
\begin{align*}
		\partial \varphi (\bar x):=\{x^*\in X^* \,|\, \langle x^*, d\rangle\leqq \varphi^\circ(\bar x; d), \ \ \forall d\in X\}.
\end{align*}

Let $S$ be a nonempty closed subset of $X$. The Clarke tangent cone to $S$ at $\bar x\in S$ is defined by
\begin{equation*}
T(\bar x; S):=\{v\in X\;|\; d_S^\circ(\bar x; v)=0\},
\end{equation*}
where $d_S$ denotes  the distance function to $S$. The Clarke normal cone to  $S$ at $\bar x\in S$ is defined by 
$$N(\bar x; S):=[T(\bar x; S)]^\circ.$$
The following lemmas will be used in the sequel.
\begin{lemma}[{see \cite[p. 52]{Clakre83}}] \label{Fermat-rule}
Let $\varphi$ be a locally Lipschitz function from $X$ to $\mathbb{R}$ and $S$ be a nonempty subset of $X$.  If $\bar x$ is a local minimizer of $\varphi$ on $S$, then
\begin{equation*}
0\in\partial \varphi(\bar x)+N(\bar x; S).
\end{equation*}
\end{lemma}
\begin{lemma}[{see \cite[Proposition 2.3.3]{Clakre83}}] \label{sum-rule}
Let $\varphi_l\colon X\to \mathbb{R}$, $l=1, \ldots, p$, $p\geqq 2$,  be locally Lipschitz around $\bar x\in X$. Then we have the following inclusion
	\begin{equation*}
	\partial (\varphi_1+\ldots+\varphi_p) (\bar x)\subset \partial  \varphi_1 (\bar x) +\ldots+\partial \varphi_p (\bar x).
	\end{equation*}
\end{lemma}
\begin{lemma}[{see \cite[Proposition 2.3.12]{Clakre83}}] \label{max-rule}
Let $\varphi_l\colon X\to {\mathbb{R}}$, $l=1, \ldots, p$,  be  locally Lipschitz around $\bar x\in X$. Then the function
	$ \phi(\cdot):=\max\{\varphi_l(\cdot)\;|\;l=1, \ldots, p\}$
	is also locally Lipschitz around $\bar x$ and one has
	\begin{equation*}
	\partial \phi(\bar x)\subset \bigcup\bigg\{\sum_{l=1}^{p}\lambda_l\partial\varphi_l (\bar x)\;|\; (\lambda_1, \ldots, \lambda_p)\in \mathbb{R}^p_+, \sum_{l=1}^{p}\lambda_l=1, \lambda_l[\varphi_l(\bar x)-\phi(\bar x)]=0\bigg\}.
	\end{equation*}
\end{lemma}

\section{Optimality conditions}  \label{section3}
Hereafter we assume that  the following assumptions are satisfied:
\begin{enumerate}[(i)]
\item $T$ is a compact topological space;
\item $X$ is separable; or $T$ is metrizable and $\partial g_t(x)$ is upper semicontinuous $(w^*)$ in $t$ for each  $x\in X$.
\end{enumerate}

Denote by $\mathbb{R}_+^{|T|}$ the set of all functions $\mu\colon T\to\mathbb{R}_+$ such that $\mu_t:=\mu(t)=0$ for all $t\in T$ except for finitely many points. The active constraint multipliers set of problem \eqref{semi-problem} at $\bar x\in \Omega$ is defined by
\begin{equation*}
A(\bar x):=\left\{\mu\in \mathbb{R}_+^{|T|}\;|\; \mu_tg_t(\bar x)=0, \ \ \forall t\in T\right\}.
\end{equation*}
For each $x\in X$, put $G(x):=\max_{t\in T} g_t(x)$ and 
$$T(x):= \left\{t\in T\;|\; g_t(x)=G(x)\right\}.$$  

Fix $\xi\in\mathbb{R}^m_+\setminus\{0\}$. The following theorem gives a necessary optimality condition of fuzzy Karush--Kuhn--Tucker  type  for  $\xi$-weakly Pareto solutions of problem \eqref{semi-problem}.  
\begin{theorem}\label{thrm31} Let $\bar x$ be a  $\xi$-weakly Pareto solution of problem \eqref{semi-problem}. If the following constraint qualification condition holds
\begin{equation}\label{Slater}\tag{$\mathcal{U}$}
\exists \bar d\in T(\bar x; \Omega)\ \ \text{such that} \ \ G^\circ(\bar x; \bar d)<0 , 
\end{equation}
then, for any $\delta>0$ small enough, there exist $x_\delta\in C\cap B(\bar x, \delta)$ and  $\lambda:=(\lambda_1, \ldots, \lambda_m)\in\mathbb{R}^m_+$ with $\sum_{i\in I}\lambda_i=1$ such that  
\begin{align*}
&0\in \sum_{i\in I} \lambda_i\partial f_i(x_\delta)+\mathbb{R}_+\mathrm{cl\,conv}\left\{\bigcup \partial g_t(x_\delta)\,|\, t\in T(x_\delta)\right\}+N(x_\delta; \Omega)+\frac{1}{\delta}\max_{i\in I}\{\xi_i\}B_{X^*},
\\
&\lambda_i\big[f_i(x_\delta)-f_i(\bar x)+\xi_i-\max_{i\in I}\{f_i(x_\delta)-f_i(\bar x)+\xi_i\}\big]=0, \ \ i\in I,
\end{align*}
where $\mathrm{cl\,conv}(\cdot)$ denotes the closed convex hull with the closure  taken in the weak$^\ast$-topology of the dual space $X^\ast$.
\end{theorem}
\begin{proof} For  each $x\in X$, put $\psi (x):=\max_{i\in I}\{f_i(x)-f_i(\bar x)+\xi_i\}.$ Then, we have  $\psi(\bar x)=\max_{i\in I}\{\xi_i\}$. Since $\bar x$ is a $\xi$-weakly Pareto solution of problem \eqref{semi-problem}, one has
\begin{equation}\label{equ:1}
\psi(x)\geqq 0, \ \ \forall x\in C.
\end{equation} 
Clearly,  $\psi$ is locally Lipschitz and bounded from below on $C$. By \eqref{equ:1}, we have
\begin{equation*}
\psi(\bar x)\leqq \inf_{x\in C}\psi(x)+\max_{i\in I}\{\xi_i\}.
\end{equation*}
By the Ekeland variational principle \cite[Theorem 1.1]{Ekeland74}, for any $\delta>0$, there exists $x_\delta\in C$ such that $\|x_\delta-\bar x\|<\delta$ and
\begin{equation*}
\psi(x_\delta)\leqq \psi(x)+\frac{1}{\delta}\max_{i\in I}\{\xi_i\}\|x- x_\delta\|, \ \ \forall x\in C.
\end{equation*}
For each $x\in C$, put 
$$\varphi(x):=\psi(x)+\frac{1}{\delta}\max_{i\in I}\{\xi_i\}\|x- x_\delta\|.$$
Then $x_\delta$ is a  minimizer of $\varphi$ on $C$. By Lemma \ref{Fermat-rule}, we have
\begin{equation}\label{equ:2}
0\in \partial \varphi(x_\delta)+N(x_\delta; C).
\end{equation}
By Lemma \ref{sum-rule} and \cite[Example 4, p. 198]{Ioffe79}, one has
\begin{equation}\label{equ:3}
\partial \varphi(x_\delta)\subset \partial \psi(x_\delta)+\frac{1}{\delta}\max_{i\in I}\{\xi_i\} B_{X^*}.
\end{equation}
Thanks to Lemma \ref{max-rule}, we have
\begin{equation}\label{equ:4}
\partial \psi(x_\delta)\subset \bigg\{\sum_{i\in I}\lambda_i\partial f_i(x_\delta)|\lambda_i\geqq 0, i\in I, \sum_{i\in I}\lambda_i=1, \lambda_i[f_i(x_\delta)-f_i(\bar x)+\xi_i-\psi(x_\delta)]=0 \bigg\}.
\end{equation}

We claim that there exists $\bar{\delta}>0$ such that  for all $x\in B(\bar x, \bar{\delta})$ there is $d\in T(x; \Omega)$ satisfying $G^\circ(x; d)<0$. Indeed, if otherwise, then there exists a sequence $\{x^k\}$ converging to $\bar x$ such that $G^\circ(x^k; d)\geqq 0$ for all $k\in\mathbb{N}$ and $d\in T(x; \Omega)$. Hence, $G^\circ(x^k; \bar d)\geqq 0$ for all  $k\in\mathbb{N}$. By the upper semicontinuity of $G^\circ(\cdot, \cdot)$, we have
\begin{equation*}
G^\circ(\bar x; \bar d)\geqq \limsup_{k\to\infty}G^\circ(x^k; \bar d)\geqq 0,
\end{equation*}
contrary to condition \eqref{Slater}.

By \cite[Theorem 6]{Hiriart-Urruty78} and \cite[Theorem 2.1]{Son09}, for each $\delta\in (0, \bar \delta)$, we have 
\begin{align*}
N(x_\delta; C)&\subset N(x_\delta; \Omega)+\mathbb{R}_+\partial G(x_\delta)
\\
&\subset N(x_\delta; \Omega)+\mathbb{R}_+\mathrm{cl\,conv}\left\{\bigcup \partial g_t(x_\delta)\;|\; t\in T(x_\delta)\right\}.
\end{align*}
Combining this with \eqref{equ:2}--\eqref{equ:4}, we obtain the desired assertion.  $\hfill\qed$
\end{proof}

\begin{remark}
{\rm By using approximate subdifferentials, Lee {\em et al.} \cite[Theorem 8.3]{Lee12} derived some necessary optimality conditions for $\xi$-weakly Pareto solutions of a convex infinite vector optimization problem. However, we are not familiar with any results on optimality conditions for $\xi$-weakly Pareto solutions of a nonconvex problem  of type \eqref{semi-problem}. Theorem  \ref{thrm31} may be the first result of this type. We also note here that when $T$ is a finite set, the result in Theorem  \ref{thrm31} is a corresponding result of Theorem 3.4 of  \cite{chuong16}, where the optimality condition was given in terms of the limiting subdifferential.

}
\end{remark}

\begin{theorem}\label{thm32}
 Let $\bar x$ be a  $\xi$-quasi-weakly Pareto solution of problem \eqref{semi-problem}. If condition \eqref{Slater}  holds at $\bar x$ and the convex hull of $\left\{\bigcup \partial g_t(\bar x)\;|\; t\in T(\bar x)\right\}$  is weak$^*$-closed, then,  there exist $\lambda:=(\lambda_1, \ldots, \lambda_m)\in\mathbb{R}^m_+$ with $\sum_{i\in I}\lambda_i=1$,  and $\mu\in A(\bar x)$ such that
\begin{align}\label{approximate-KKT}
0\in \sum_{i\in I}\lambda_i\partial f_i(\bar x)+\sum_{t\in T}\mu_t\partial g_t(\bar x)+\sum_{i\in I}\lambda_i\xi_i B_{X^*}+N(\bar x; \Omega).
\end{align}
\end{theorem}
\begin{proof} 
 For  each $x\in X$, put $\Phi (x):=\max_{i\in I}\{f_i(x)-f_i(\bar x)+\xi_i\|x-\bar x\|\}.$ Then, $\Phi(\bar x)=0$.  Since $\bar x$ is a $\xi$-quasi-weakly Pareto solution of problem \eqref{semi-problem}, we have 
 \begin{equation*} 
 \Phi(x)\geqq 0, \ \ \forall x\in C.
 \end{equation*} 
This means that $\bar x$ is a minimizer of $\Phi$ on $C$. By Lemma \ref{Fermat-rule}, one has  
\begin{equation}\label{equ:5}
0\in\partial \Phi(\bar x)+N(\bar x; C).
\end{equation}
By Lemma \ref{max-rule}, we have
\begin{equation}\label{equ:6}
\partial \Phi(\bar x)\subset \left\{  \sum_{i\in I}\lambda_i\left[\partial f_i(\bar x)+\xi_iB_{X^*}\right]\;|\; \lambda_i\geqq 0, \sum_{i\in I}\lambda_i=1\right\}. 
\end{equation}
Since condition \eqref{Slater}  holds at $\bar x$ and the convex hull of $\left\{\bigcup \partial g_t(\bar x)\;|\; t\in T(\bar x)\right\}$  is weak$^*$-closed, we obtain
\begin{equation}\label{equ:7}
\begin{split}
N(\bar x; C)&\subset N(\bar x; \Omega)+\mathbb{R}_+\partial G(\bar x)
\\
&\subset N(\bar x; \Omega)+\mathbb{R}_+\mathrm{conv}\left\{\bigcup \partial g_t(\bar x)\;|\; t\in T(\bar x)\right\}.
\end{split}
\end{equation}
To finish the proof of the theorem, it remains to combine \eqref{equ:5}, \eqref{equ:6}, and \eqref{equ:7}. $\hfill\Box$
\end{proof}
\begin{remark}\label{Remark32}
{\rm 
\begin{enumerate}[(i)]
\item When $X$ is a finite-dimensional space and the constraint functions $g_t\colon X\to\mathbb{R},$ $t\in T$, are locally Lipschitz with respect to $x$ uniformly in $t\in T$, i.e., for each $x\in X$, there is a neighborhood $U$ of $x$ and a constant $K>0$ such that
\begin{equation*}
|g_t(u)-g_t(v)|\leqq K\|u-v\|, \ \ \forall u, v \in U \ \ \text{and}\ \ \forall t\in T,
\end{equation*}
then the set $\left\{\bigcup \partial g_t(x)\;|\; t\in T(x)\right\}$ is compact. Consequently, its convex hull is always closed. 
\item Recently, by using the Chankong--Haimes scalarization scheme (see \cite{Chankong83}), Kim and Son \cite[Theorem 3.3]{Son18} obtained some necessary optimality conditions for  $\xi$-quasi Pareto solutions of a locally Lipschitz semi-infinite vector optimization problem.  We note here that condition \eqref{Slater} is weaker than the following condition \eqref{Kim-Son} used in \cite[Theorem 3.3]{Son18}:  
\begin{equation}\label{Kim-Son}\tag{${\mathcal{A}_i}$}
\exists \bar d\in T(\bar x; \Omega): g_t^\circ(\bar x; \bar d)<0  \ \ \forall t\in T(\bar x)  \ \ \text{and} \ \ f_k^\circ(\bar x; \bar d)<0,  \ \ \forall k\in I\setminus\{i\}.
\end{equation}
Thus, Theorem \ref{thm32} improves \cite[Theorem 3.3]{Son18}. To illustrate, we consider the following example  
\end{enumerate}
}
\end{remark} 
\begin{example}\label{Example31}
{\rm Let $f\colon\mathbb{R}\to\mathbb{R}^2$   be defined by $f(x):=(f_1(x),  f_2(x))$, where
\begin{equation*}
f_1(x):=
\begin{cases}
x^2\cos\frac{1}{x},\ \ &\text{if}\ \ x\neq 0,
\\
0,\ \ &\text{otherwise},
\end{cases}
\end{equation*}
and $f_2(x):=0$ for all $x\in\mathbb{R}$. Assume that $\Omega=\mathbb{R}$, $T=[1, 2]$, and $g_t(x)=tx$ for all $x\in\mathbb{R}$. Then the feasible set of problem \eqref{semi-problem} is $C=(-\infty, 0]$. Let $\bar x:=0\in C$. Clearly, for any $\xi\in \mathbb{R}^2_+\setminus\{0\}$, $\bar x$ is a $\xi$-quasi-weakly Pareto solution of \eqref{semi-problem}. It is easy to check that
\begin{equation*}
\partial f_1(\bar x)=[-1, 1], \partial f_2(\bar x)=\{0\}, \ \ \text{and} \ \ \partial g_t(\bar x)=\{t\}, \ \ \ \forall t\in T.
\end{equation*}
Hence, for each $d\in\mathbb{R}$, we have
\begin{equation*}
f_1^\circ(\bar x; d)=|d|, f_2^\circ(\bar x; d)=0,\ \ \text{and} \ \ g_t^\circ(\bar x; d)=td, \ \ \ \forall t\in T.
\end{equation*}
Clearly, every $d<0$ satisfies condition \eqref{Slater}. However, for every $i=1, 2$, condition \eqref{Kim-Son} does not hold. Thus  Theorem \ref{thm32} can be applied for this example, but not   \cite[Theorem 3.3]{Son18}.
}
\end{example}

The following concept is inspired from \cite{chuong14}.
\begin{definition}
{\rm Let $f:=(f_1, \ldots, f_m)$ and $g_T:=(g_t)_{t\in T}$.
\begin{enumerate}[(i)]
\item We say that $(f, g_T)$ is {\em generalized convex} on $\Omega$ at $\bar x$ if, for any $x\in \Omega$, $z_i^*\in \partial f_i(\bar x)$, $i=1, \ldots, m$, and $x^*_t\in \partial g_t(\bar x)$, $t\in T$, there exists $\nu\in T(\bar x; \Omega)$ satisfying
\begin{align*}
&f_i(x)-f_i(\bar x)\geqq \langle z^*_i, \nu\rangle, \ \ i=1, \ldots, m,
\\
&g_t(x)-g_t(\bar x)\geqq \langle x^*_t, \nu\rangle, \ \ t\in T,
\end{align*}  
and 
\begin{equation*}
\langle b^*, \nu\rangle\leqq \|x-\bar x\|, \ \ \forall b^*\in B_{X^*}. 
\end{equation*}
\item We say that $(f, g_T)$ is {\em strictly generalized convex} on $\Omega$ at $\bar x$ if, for any $x\in \Omega\setminus\{\bar x\}$, $z_i^*\in \partial f_i(\bar x)$, $i=1, \ldots, m$, and $x^*_t\in \partial g_t(\bar x)$, $t\in T$, there exists $\nu\in T(\bar x; \Omega)$ satisfying
\begin{align*}
&f_i(x)-f_i(\bar x)>\langle z^*_i, \nu\rangle, \ \ i=1, \ldots, m,
\\
&g_t(x)-g_t(\bar x)\geqq \langle x^*_t, \nu\rangle, \ \ t\in T,
\end{align*}
\end{enumerate}
and 
\begin{equation*}
\langle b^*, \nu\rangle\leqq \|x-\bar x\|, \ \ \forall b^*\in B_{X^*}. 
\end{equation*}
}
\end{definition}
\begin{remark}
{\rm Clearly, if $\Omega$ is convex and $f_i$, $i\in I$, and $g_t$, $t\in T$ are convex (resp., strictly convex), then $(f, g_T)$ is generalized convex (resp., strictly generalized convex) on $\Omega$ at any $\bar x\in\Omega$ with $\nu:=x-\bar x$ for each $x\in\Omega$. Furthermore, by a  similar argument  in \cite[Example 3.2]{chuong14}, we can show that the class of generalized convex functions is
properly larger than the one of convex functions. 

}
\end{remark} 
\begin{theorem}\label{thm33}
 Let $\bar x\in C$ and assume that there exist  $\lambda:=(\lambda_1, \ldots, \lambda_m)\in\mathbb{R}^m_+$ with $\sum_{i\in I}\lambda_i=1$,  and $\mu\in A(\bar x)$  satisfying \eqref{approximate-KKT}.
\begin{enumerate}[\rm(i)]
\item If $(f, g_T)$ is generalized convex on $\Omega$ at $\bar x$, then $\bar x$ is a  $\xi$-quasi-weakly Pareto solution of \eqref{semi-problem}. 

\item If $(f, g_T)$ is strictly generalized convex on $\Omega$ at $\bar x$, then $\bar x$ is a $\xi$-quasi Pareto solution of \eqref{semi-problem}.   
\end{enumerate}  
\end{theorem}
\begin{proof} We will follow the proof scheme of \cite[Theorem 3.13]{chuong16}. By \eqref{approximate-KKT}, there exist $z_i^*\in\partial f_i(\bar x)$, $i=1, \ldots, m$,  $x^*_t\in\partial g_t(\bar x)$, $t\in T$, $b^*\in B_{X^*}$, and $\omega^*\in N(\bar x; \Omega)$ such that
\begin{equation*}
\sum_{i\in I}\lambda_iz^*_i+\sum_{t\in T}\mu_tx^*_t+\sum_{i\in I}\lambda_i\xi_i b^*+\omega^*=0,
\end{equation*}
or, equivalent,
\begin{equation*}
\sum_{i\in I}\lambda_iz^*_i+\sum_{t\in T}\mu_tx^*_t+\sum_{i\in I}\lambda_i\xi_i b^*=-\omega^*.
\end{equation*}

We first prove (i). On the contrary, if $\bar x$ is not a  $\xi$-quasi-weakly Pareto solution of \eqref{semi-problem}, then there is $x\in C$ such that
\begin{equation*}
f(x)+\|x-\bar x\|\xi< f(\bar x).
\end{equation*} 
From this and the fact that $\sum_{i\in I}\lambda_i=1$, we obtain
\begin{equation}\label{equ:8}
\sum_{i\in I} \lambda_i f_i(x) +\sum_{i\in I} \lambda_i \xi_i\|x-\bar x\|< \sum_{i\in I} \lambda_i f_i(\bar x).
\end{equation} 
Since $(f, g_T)$ is generalized convex on $\Omega$ at $\bar x$, for such $x$, there exists $\nu\in T(\bar x; \Omega)$ such that
\begin{align*}
0\leqq -\langle \omega^*, \nu\rangle&= \sum_{i\in I}\lambda_i \langle z^*_i, \nu \rangle  +\sum_{t\in T}\mu_t \langle x^*_t, \nu \rangle  +\sum_{i\in I}\lambda_i\xi_i \langle b^*, \nu\rangle  
\\
&\leqq \sum_{i\in I}\lambda_i [f_i(x)-f_i(\bar x)]+\sum_{t\in T}\mu_t [g_t(x)-g_t(\bar x)]+\sum_{i\in I}\lambda_i\xi_i\|x-\bar x\|.
\end{align*}
Hence,
\begin{equation*}
\sum_{i\in I} \lambda_i f_i(\bar x)\leqq \sum_{i\in I} \lambda_i f_i(x)+\sum_{t\in T}\mu_t [g_t(x)-g_t(\bar x)]+\sum_{i\in I}\lambda_i\xi_i\|x-\bar x\|.
\end{equation*}
Combining this with the facts that $x, \bar x\in C$ and $\mu\in A(\bar x)$, we conclude that
\begin{equation}\label{equ:9n}
\sum_{i\in I} \lambda_i f_i(\bar x)\leqq \sum_{i\in I} \lambda_i f_i(x) +\sum_{i\in I}\lambda_i\xi_i\|x-\bar x\|,
\end{equation}
contrary to \eqref{equ:8}.

We now prove (ii). Assume on the contrary that $\bar x$ is not a $\xi$-quasi Pareto solution of \eqref{semi-problem}, i.e., there exists $y\in C$ satisfying
\begin{equation*}
f(y)+\|y-\bar x\|\xi\leq f(\bar x).
\end{equation*} 
This implies that $y\neq \bar x$  and 
\begin{equation}\label{equ:9}
\sum_{i\in I} \lambda_i f_i(y) +\sum_{i\in I} \lambda_i \xi_i\|y-\bar x\|\leqq \sum_{i\in I} \lambda_i f_i(\bar x).
\end{equation}
Since $(f, g_T)$ is strictly generalized convex on $\Omega$ at $\bar x$, for such $y$, there exists $\vartheta\in T(\bar x; \Omega)$ such that
\begin{align*}
0\leqq -\langle \omega^*, \vartheta\rangle&= \sum_{i\in I}\lambda_i \langle z^*_i, \vartheta \rangle  +\sum_{t\in T}\mu_t \langle x^*_t, \vartheta \rangle  +\sum_{i\in I}\lambda_i\xi_i \langle b^*, \vartheta\rangle  
\\
&< \sum_{i\in I}\lambda_i [f_i(y)-f_i(\bar x)]+\sum_{t\in T}\mu_t [g_t(y)-g_t(\bar x)]+\sum_{i\in I}\lambda_i\xi_i\|y-\bar x\|.
\end{align*}
An analysis similar to that in the proof of \eqref{equ:9n} shows that
\begin{equation*} 
\sum_{i\in I} \lambda_i f_i(\bar x)< \sum_{i\in I} \lambda_i f_i(y) +\sum_{i\in I}\lambda_i\xi_i\|y-\bar x\|,
\end{equation*}
contrary to \eqref{equ:9}. $\hfill\Box$
\end{proof}
\begin{remark}
{\rm
The conclusions of Theorem \ref{thm33} are still valid if $(f, g_T)$ is generalized convex in the sense of Chuong and Kim \cite[Definition 3.3]{chuong14}.
}
\end{remark}

We now present an example which demonstrates the importance
of the generalized convexity of $(f, g_T)$ in Theorem \ref{thm33}. In particular, condition \eqref{approximate-KKT} alone is not  sufficient to guarantee that $\bar x$ is a $\xi$-quasi-weakly Pareto solution of \eqref{semi-problem} if the generalized convexity of $(f, g_T)$ on $\Omega$ at $\bar x$ is violated.
\begin{example}
{\rm  Let $f\colon\mathbb{R}\to\mathbb{R}^2$   be defined by $f(x):=(f_1(x),  f_2(x))$, where
\begin{equation*}
f_i(x):=
\begin{cases}
x^2\cos\frac{1}{x},\ \ &\text{if}\ \ x\neq 0,
\\
0,\ \ &\text{otherwise},
\end{cases}
\end{equation*}
for $i=1, 2$. Assume that $\Omega=\mathbb{R}$, $T=[0, 1]$, and $g_t(x)=x-t$  for all $x\in \mathbb{R}$ and $t\in [0, 1]$.  Then, the feasible set of \eqref{semi-problem} is $C=(-\infty, 0]$. Let $\bar x:=0\in C$. Clearly, $f_i$, $i=1, 2$, and $g_t$, $t\in T$, are locally Lipschitz at $\bar x$. An easy computation shows that 
\begin{equation*}
\partial f_1(\bar x)=\partial f_2(\bar x)=[-1, 1], \ \ \text{and}\ \ \partial g_t(\bar x)=\{1\}, \ \ \ \forall t\in T.
\end{equation*}
Take arbitrarily $\xi=(\xi_1, \xi_2)\in\mathbb{R}^2_+\setminus\{0\}$ satisfying $\xi_i<\frac{1}{\pi}$ for all $i=1, 2$. Then  we see that $\bar x$ satisfies condition \eqref{approximate-KKT} for $\lambda_1=\lambda_2=\frac{1}{2}$, and $\mu_t=0$ for all $t\in T$. However, $\bar x$ is not a $\xi$-quasi-weakly Pareto solution of \eqref{semi-problem}. Indeed, let $\hat{x}=-\frac{1}{\pi}\in C$. Then,
\begin{equation*}
f_i(\hat{x})+\xi_i\|\hat{x}-\bar x\|=\frac{1}{\pi}\left(\xi_i-\frac{1}{\pi}\right)<f_i(\bar x), \ \ \ \forall i=1, 2,
\end{equation*}
as required. We now show that $(f, g_T)$ is not generalized convex on $\Omega$ at $\bar x$. Indeed, by choosing $z_i^*=0\in \partial f_i(\bar x)$ for $i=1, 2$, we have
\begin{equation*}
f_i(\hat{x})-f_i(\bar x)=-\frac{1}{\pi^2}<\langle z^*_i, \nu\rangle, \ \ \ \forall \nu\in \mathbb{R}.
\end{equation*}
}
\end{example}

\section{Applications}\label{Applications}
\subsection{Cone-constrained convex vector optimization problems}

In this subsection,  we consider the following cone-constrained convex vector optimization problem:
\begin{equation}\label{semi-problem-b}\tag{CCVP}
\begin{split}
&\text{\rm Min}_{\,\mathbb{R}^m_+}\,f(x):=(f_1(x), \ldots, f_m(x))
\\
&\text{s.t.}\ \ x\in C:=\{x\in \Omega \;|\; g(x)\in -K\},  
\end{split} 
\end{equation}
where the function $f$ and the set $\Omega$ are as in the previous sections, $K$ is a closed convex cone in a normed space $Y$, and $g$ is a  continuous and $K$-convex mapping  from $X$  to $Y$. Recall that the mapping $g$ is said to be $K$-convex if 
$$g[\theta x+(1-\theta)y]-\theta g(x)-(1-\theta)g(y) \in -K,\ \ \forall x,y \in X,\ \ \forall \theta \in [0,1].$$
Let $Y^\ast$ be the dual space of $Y$ and $K^+$ be the positive polar cone of $K$, i.e.,
$$K^+:=\{y^\ast\in Y^\ast \mid \langle y^\ast, k\rangle \geqq 0,\ \ \forall k \in K \}.$$
Then, $K^+$ is weak$^\ast$-closed. Moreover, it is easily seen that
$$g(x) \in -K \Leftrightarrow  g_s(x)  \leqq 0,\ \  \forall  s \in K^+,$$
where $g_s(x):=\langle s, g(x)\rangle$.
 Hence,  problem \eqref{semi-problem-b} is equivalent to the following vector optimization problem: 
\begin{equation}\label{semi-problem-c}
\begin{split}
&\text{\rm Min}_{\,\mathbb{R}^m_+}\,f(x):=(f_1(x), \ldots, f_m(x))
\\
&\text{s.t.}\ \ x\in C:=\{x\in \Omega \;|\; g_s(x) \leqq 0,\ \ s\in K^+\}.  
\end{split} 
\end{equation}
In order to apply the results in Section \ref{section3} to  problem \eqref{semi-problem-c}, we need to have a compact set of indices, which is not the case with the cone $K^+$.  However, since $Y$ is a normed space, it is easily seen that the set $K^+\cap B_{Y^\ast}$ is weak$^\ast$-compact and  
$$g(x) \in -K \Leftrightarrow g_s(x)\leqq 0,\ \ \forall s \in K^+ \cap B_{Y^\ast}.$$
Hence we can rewrite problem \eqref{semi-problem-c} as  
\begin{equation*}\label{semi-problem-c1}
\begin{split}
&\text{\rm Min}_{\,\mathbb{R}^m_+}\,f(x):=(f_1(x), \ldots, f_m(x))
\\
&\text{s.t.}\ \ x\in C:=\{x\in \Omega \;|\; g_s(x) \leqq 0,\ \ s\in {T}\},  
\end{split} 
\end{equation*}
where ${T}:= K^+ \cap B_{Y^\ast}.$

\begin{proposition}\label{thm41} Let $\xi\in\mathbb{R}^m_+\setminus\{0\}$ and $\bar x$ be a  $\xi$-quasi-weakly Pareto solution of problem \eqref{semi-problem-b}. If  condition \eqref{Slater}  holds at $\bar x$ and the convex hull  ${\rm co}\left\{\bigcup \partial g_s(\bar x)\;|\; s\in  {I}(\bar x)\right\}$  is weak$^*$-closed, then,  there exist $\lambda:=(\lambda_1, \ldots, \lambda_m)\in\mathbb{R}^m_+$ with $\sum_{i\in I}\lambda_i=1$, and $\zeta \in K^+$ such that
\begin{align}\label{approximate-KKT-b}
0\in \sum_{i\in I}\lambda_i\partial f_i(\bar x)+\partial  g_\zeta(\bar x)+\sum_{i\in I}\lambda_i\xi_i B_{X^*}+N(\bar x; \Omega).
\end{align}
\end{proposition}
\begin{proof} By  Theorem \ref{thm32}, there exist $\lambda:=(\lambda_1, \ldots, \lambda_m)\in\mathbb{R}^m_+$ with $\sum_{i\in I}\lambda_i=1$,  and $\mu\in  {A}(\bar x)$ such that
\begin{align*}
0\in \sum_{i\in I}\lambda_i\partial f_i(\bar x)+\sum_{s \in \mathcal{T}}\mu_s\partial g_s(\bar x)+\sum_{i\in I}\lambda_i\xi_i B_{X^*}+N(\bar x; \Omega).
\end{align*}
Note that for each $s \in  {T}$, the function $g_s$ is continuous and convex on $X$. Moreover, since $\mu \in {A}(\bar{x})$, there exists only finitely many $\mu_s, s \in {T},$ differ from zero. Hence,   
$$\sum_{s \in  {T}}\mu_s\partial g_s(\bar x)=\partial\left(\left\langle\sum_{s \in {T}}\mu_s s, g\right\rangle\right)(\bar x)=\partial g_\zeta(\bar x),$$
where $\zeta:=\sum_{s \in  {T}}\mu_s s$. Clearly, $\zeta \in K^+$. The proof is completed. $\hfill\Box$
\end{proof}
\begin{remark}{\rm  Assume that $\Omega$ is convex. Since $g$ is  continuous and  $K$-convex, we see that if  $f_i, i\in I$, are convex (resp., strictly convex), then $(f, g_T)$ is generalized convex (resp., strictly generalized convex) on $\Omega$ at any $\bar x\in \Omega$ with $\nu:=x-\bar x$ for each $x\in \Omega$. Thus, by Theorem \ref{thm33},   \eqref{approximate-KKT-b} is a sufficient condition for a point  $\bar{x}\in C$ to be a $\xi$-quasi-weakly Pareto (resp., $\xi$-quasi  Pareto) solution of  \eqref{semi-problem-b} provided that $f_i, i\in I$, are convex (resp., strictly convex).
}  
\end{remark}
\subsection{Semidefinite vector optimization problem}

Let $f : \Bbb{R}^n \to \Bbb{R}$ be a locally Lipschitz continuous function, $\Omega \subset \Bbb{R}^n$ be a nonempty closed convex subset of $\Bbb{R}^n$ and $g : \Bbb{R}^n \to S^p$ be a continuous mapping, where $S^p$ denotes the set of  $p\times p$ symmetric matrices. For a $ p\times p$ matrix $A=(a_{ij})$, the notion ${\rm trace}(A)$ is defined by
 $${\rm trace} (A):= \sum_{i=1}^p a_{ii}.$$  
We suppose that $S^p$ is equipped with a scalar product $A \bullet B:={\rm trace}(AB)$, where $AB$ is the matrix product  of $A$ and $B$. A matrix $A\in S^p$ is said to be a negative semidefinite (resp., positive semidefinite) matrix  if $\langle v, Av\rangle \leqq 0$ (resp., $\langle v, Av\rangle \geqq 0$) for all $v \in \Bbb{R}^n$. If matrix $A$ is negative semidefinite (resp., positive semidefinite matrix), it is denoted by  $A \preceq 0$  (resp., $A \succeq 0$). 

We now consider the following semidefinite vector optimization problem:
\begin{equation}\label{semi-problem-d}\tag{SDVP}
\begin{split}
&\text{\rm Min}_{\,\mathbb{R}^m_+}\,f(x):=(f_1(x), \ldots, f_m(x))
\\
&\text{s.t.}\ \ x\in C:=\{x\in \Omega \;|\; g(x) \preceq 0\}.  
\end{split} 
\end{equation}
Let us denote by $S^p_+$ the set of all positive semidefinite matrices of  $S^p$. It is well known that  $S^p_+$ is a proper convex cone, i.e., it is closed, convex, pointed, and solid (see \cite{Gri}) and that $S^p_+$ is a self-dual cone, i.e.,  $(S^p_+)^+=S^p_+$.  Hence,  problem \eqref{semi-problem-d} can be rewritten under the form of problem \eqref{semi-problem-b}, where $K:=S^p_+$.   For  $\Lambda \in K$, the function $g_\Lambda$ becomes a function from $\Bbb{R}^n$ to $\Bbb{R}$ defined by
$$ g_\Lambda(x)=\Lambda \bullet g(x),\ \ \forall x \in \Bbb{R}^n.$$

If $g$ is affine, i.e., $g(x):=F_0+\sum_{i=1}^n F_ix_i$, where $F_0, F_1, \ldots, F_n \in S^p$ are given matrices, then the subdifferential of the function $g_\Lambda(x)$ is equal to
$$\partial (g_\Lambda)(x)=(\Lambda \bullet F_1, \ldots, \Lambda \bullet F_n)=:(\Lambda \bullet F).$$  
In that case, by Remark \ref{Remark32},  the second assumption of Proposition \ref{thm41} can be removed. Thus we obtain the following result. 

\begin{proposition} Let $\xi\in\mathbb{R}^m_+\setminus\{0\}$.  If $\bar x$ is a  $\xi$-quasi-weakly Pareto solution of problem \eqref{semi-problem-d}, then,  there exist $\lambda:=(\lambda_1, \ldots, \lambda_m)\in\mathbb{R}^m_+$ with $\sum_{i\in I}\lambda_i=1$, and $\Lambda \in S^p_+$ such that
\begin{align}\label{approximate-KKT-d}
0\in \sum_{i\in I}\lambda_i\partial f_i(\bar x)+\Lambda \bullet F+\sum_{i\in I}\lambda_i\xi_i B_{X^*}+N(\bar x; \Omega).
\end{align}
\end{proposition}

\begin{remark} Since $g$ is an affine mapping, by Theorem \ref{thm33}, \eqref{approximate-KKT-d} is a sufficient condition for a point  $\bar{x}\in C$ to be a $\xi$-quasi-weakly Pareto (resp., $\xi$-quasi  Pareto) solution of  \eqref{semi-problem-d} provided that  $f_i, i\in I$, are convex (resp., strictly convex). 
\end{remark}

\section*{Acknowledgments} The authors would like to thank the anonymous referee and the handling Associate Editor for their valuable remarks and detailed suggestions that allowed us to improve the original version. The research of Ta Quang Son was supported by Vietnam National Foundation for Science and Technology Development (NAFOSTED) under grant number 101.01-2017.08. The research of Nguyen Van Tuyen was supported by the Ministry of Education and Training of Vietnam [grant number B2018-SP2-14].  The research of Ching-Feng Wen was supported by the Taiwan MOST [grant number 107-2115-M-037-001] as well as the grant from Research Center for Nonlinear 	Analysis and Optimization, Kaohsiung Medical University, Taiwan. 


\end{document}